\documentclass{amsart}
\usepackage{amsfonts}
\usepackage{amssymb}
\usepackage{amsmath}

\newtheorem{theorem}{Theorem}[section]
\theoremstyle{plain}

\newtheorem{corollary}[theorem]{Corollary}

\newtheorem{example}{Example}[section]

\newtheorem{remark}[theorem]{Remark}

\numberwithin{equation}{section}
\input{tcilatex}

\begin{document}
\title[\textbf{generalizations on the univalence of an integral operator}]{%
\textbf{Some generalizations} \textbf{on the univalence of an integral
operator and quasiconformal extensions}}
\author{Murat \c{C}A\u{G}LAR}
\address{Department of Mathematics, Faculty of Science, Ataturk University,
Erzurum, 25240, Turkey.}
\email{mcaglar@atauni.edu.tr}
\author{Halit ORHAN}
\address{Department of Mathematics, Faculty of Science, Ataturk University,
Erzurum, 25240, Turkey.}
\email{E-mail: horhan@atauni.edu.tr}

\begin{abstract}
By using the method of Loewner chains, we establish some sufficient
conditions for the analyticity and univalency of functions defined by an
integral operator. Also, we refine the result to a quasiconformal extension
criterion with the help of Beckers's method.
\end{abstract}

\maketitle

\section{Introduction}

Let $\mathcal{A}$ the class of functions $f$ which are analytic in the open
unit disk $\mathcal{U=}\left\{ z\in 
\mathbb{C}
:\left\vert z\right\vert <1\right\} $ with $f(0)=f^{\prime }(0)-1=0.$ We
denote by $\mathcal{U}_{r}$ the open disk $\left\{ z\in {\mathbb{C}}%
:\left\vert z\right\vert <r\right\} $ $,$ where $0<r\leq 1$, by $\mathcal{U}=%
\mathcal{U}_{1}$ the open unit disk of the complex plane and by $I$ the
interval $[0,\infty )$.

Let $k$ be constant in $[0,1).$ Then a homeomorphism $f$ of $G\subset 
\mathbb{C}$ is said to be $k-$\textit{quasiconformal, }if $\partial _{z}f$
and $\partial _{\overline{z}}f$ in the distributional sense are locally
integrable on $G$ and fulfill the inequality $\left\vert \partial _{%
\overline{z}}f\right\vert \leq k\left\vert \partial _{z}f\right\vert $
almost everywhere in $G.$ If we do not need to specify $k,$ we will simply
call that $f$ is \textit{quasiconformal.}

Three of the most important and known univalence criteria for analytic
functions defined in the open unit disk were obtained by Nehari \cite{Neh},
Ozaki-Nunokawa \cite{Oz-Nu} and Becker \cite{Bec}. Some extensions of these
three criteria were given by (see [\cite{Ove}, \cite{Radu1}, \cite{Radu3}, 
\cite{Radu4}, \cite{Radu5} and \cite{Tudor}]).During the time, a lot of
univalence criteria were obtained by different authors (see also \cite%
{Deniz1}, \cite{Deniz2} and \cite{Gol}).

In the present investigation, we will obtain a number of new criteria for
the functions defined by the integral operator $\mathcal{F}_{\beta }(z).$
Also, we obtain a refinement to a quasiconformal extension criterion of the
main result.

\section{Preliminaries}

Before proving our main theorem we need a brief summary of the method of
Loewner chains and quasiconformal extension criterion.

A function $\mathcal{L}(z,t):\mathcal{U}\times \lbrack 0,\infty )\rightarrow 
\mathbb{C}$ is said to be \textit{subordination chain} \textit{(or} \textit{%
Loewner chain)} if:

\begin{itemize}
\item[(i)] $\mathcal{L}(z,t)$ is analytic and univalent in $\mathcal{U}$ for
all $t\geq 0$.

\item[(ii)] $\mathcal{L}(z,t)\prec \mathcal{L}(z,s)$ for all $0\leq t\leq
s<\infty $, where the symbol $"\prec "$ stands for subordination.
\end{itemize}

In proving our results, we will need the following theorem due to Ch.
Pommerenke \cite{Pom}.

\begin{theorem}
\label{t1}\textbf{\ }\textit{Let }$\mathcal{L}%
(z,t)=a_{1}(t)z+a_{2}(t)z^{2}+...,\;a_{1}(t)\neq 0$\textit{\ be analytic in }%
$\mathcal{U}_{r}$\textit{\ for all }$t\in I,$\textit{\ locally absolutely
continuous in }$I,$\textit{\ and locally uniform with respect to }$\mathcal{U%
}_{r}$\textit{$.$ For almost all }$t\in I,$\textit{\ suppose that}%
\begin{equation}
z\frac{\partial \mathcal{L}(z,t)}{\partial z}=p(z,t)\frac{\partial \mathcal{L%
}(z,t)}{\partial t},\text{ }\forall z\in \mathcal{U}_{r}  \label{1.1}
\end{equation}%
\textit{where }$p(z,t)$\textit{\ is analytic in }$\mathcal{U}$\textit{\ and
satisfies the condition }$\Re p(z,t)>0$\textit{\ for all }$z\in \mathcal{U}%
,\;t\in I.$\textit{\ If }$\left\vert a_{1}(t)\right\vert \rightarrow \infty $%
\textit{\ for }$t\rightarrow \infty $\textit{\ and }$\{\mathcal{L}%
(z,t)\diagup a_{1}(t)\}$\textit{\ forms a normal family in }$\mathcal{U}%
_{r}, $\textit{\ then for each }$t\in I,$\textit{\ the function }$\mathcal{L}%
(z,t)$\textit{\ has an analytic and univalent extension to the whole disk }$%
\mathcal{U}.$
\end{theorem}

The method of constructing quasiconformal extension criteri\textbf{a} is
based on the following result due to Becker (see \cite{Bec}, \cite{Be2} and
also \cite{Be3}).

\begin{theorem}
\label{t2}Suppose that $\mathcal{L}(z,t)$ is a Loewner chain for which $%
p(z,t)$ given in (\ref{1.1}) satisfies the condition%
\begin{eqnarray*}
p(z,t) &\in &\mathcal{U}(k):=\left\{ w\in {\mathbb{C}}:\left\vert \frac{w-1}{%
w+1}\right\vert \leq k\right\} \\
&=&\left\{ w\in {\mathbb{C}}:\left\vert w-\frac{1+k^{2}}{1-k^{2}}\right\vert
\leq \frac{2k}{1-k^{2}}\right\} ,\ \ \ \left( 0\leq k<1\right)
\end{eqnarray*}%
for all $z\in \mathcal{U}$ and $t\geq 0.$ Then $\mathcal{L}(z,t)$ admits a
continuous extension to $\overline{\mathcal{U}}$ for each $t\geq 0$ and the
function $F(z,\overline{z})$ defined by%
\begin{equation*}
F(z,\overline{z})=\left\{ 
\begin{array}{c}
\mathcal{L}(z,0),\text{ \ \ \ \ \ \ \ if \ }\left\vert z\right\vert <1 \\ 
\mathcal{L}(\frac{z}{\left\vert z\right\vert },\log \left\vert z\right\vert
),\text{ \ if \ }\left\vert z\right\vert \geq 1%
\end{array}%
\right.
\end{equation*}%
is a $k-$quasiconformal extension of $\mathcal{L}(z,0)$ to ${\mathbb{C}}.$
\end{theorem}

Examples of quasiconformal extension criteri\textbf{a} can be found in \cite%
{Ah}, \cite{AnHi}, \cite{Bet}, \cite{Kr}, \cite{Pf} and more recently in 
\cite{Ho1}, \cite{Ho2}, \cite{Ho3}.

\section{Main Results}

In this section, making use of Theorem \ref{t1}, we obtain certain
sufficient conditions for univalence of an integral operator.

\begin{theorem}
\label{t3}\textit{Let }$m$\textit{\ be a positive real number and let }$%
\alpha ,$ $\beta $\textit{\ be complex numbers } such that $\Re \alpha
<1/2,\;\Re \beta >0$ and$\;f\in \mathcal{A}.$\textit{\ Let} $g$\textit{\ and 
}$h$\textit{\ be two analytic functions in }$\mathcal{U},$ $%
g(z)=1+b_{1}z+...,$ $h(z)=c_{0}+c_{1}z+...$\textit{. If the following
inequalities }%
\begin{equation}
\left\vert \frac{f^{\prime }(z)}{g(z)-\alpha }-\frac{m-1}{2}\right\vert <%
\frac{m+1}{2},  \label{2.22}
\end{equation}%
\textit{and}%
\begin{equation*}
\left\vert \left( \frac{f^{\prime }(z)}{g(z)-\alpha }-1\right) \left\vert
z\right\vert ^{\beta (m+1)}\right.
\end{equation*}%
\begin{equation*}
+\left( 1-\left\vert z\right\vert ^{\beta \left( m+1\right) }\right) \left[
2z^{\beta }\frac{f^{\prime }(z)h(z)}{g(z)-\alpha }+\frac{1}{\beta }\frac{%
zg^{\prime }(z)}{g(z)-\alpha }\right]
\end{equation*}%
\begin{equation}
\left. +\frac{z^{\beta +1}\left( 1-\left\vert z\right\vert ^{\beta \left(
m+1\right) }\right) ^{2}}{\left\vert z\right\vert ^{\beta \left( m+1\right) }%
}\left[ \frac{z^{\beta -1}f^{\prime }(z)h^{2}(z)}{g(z)-\alpha }+\frac{1}{%
\beta }\left( \frac{g^{\prime }(z)h(z)}{g(z)-\alpha }-h^{\prime }(z)\right) %
\right] -\frac{m-1}{2}\right\vert \leq \frac{m+1}{2}  \label{2.2}
\end{equation}%
are true for all $z\in \mathcal{U},$ then the function $\mathcal{F}_{\beta
}(z)$ defined by%
\begin{equation}
\mathcal{F}_{\beta }(z)=\left[ \beta \dint\limits_{0}^{z}u^{\beta
-1}f^{\prime }(u)du\right] ^{1/\beta }  \label{2.3}
\end{equation}%
is analytic and univalent\textit{\ in }$\mathcal{U},$\textit{\ where the
principal branch is intended.}
\end{theorem}

\begin{proof}
We shall prove that there exists a real number $r,$ $r\in (0,1]$ such that
the function $\mathcal{L}:\mathcal{U}_{r}\times I\rightarrow 
\mathbb{C}
,$ defined formally by%
\begin{equation}
\mathcal{L}(z,t)=\left[ \beta \dint\limits_{0}^{e^{-t}z}u^{\beta
-1}f^{\prime }(u)du+\frac{\left( e^{\beta mt}-e^{-\beta t}\right) z^{\beta
}\left( g\left( e^{-t}z\right) -\alpha \right) }{1+\left( e^{\beta
mt}-e^{-\beta t}\right) z^{\beta }h\left( e^{-t}z\right) }\right] ^{1/\beta }
\label{2.4}
\end{equation}%
is analytic in $\mathcal{U}_{r}$ for all $t\in I.$

Because $f\in \mathcal{A}$ we have 
\begin{equation*}
f(z)=z+a_{2}z^{2}+...+a_{n}z^{n}+...,\text{ \ \ }\forall z\in \mathcal{U}.
\end{equation*}%
Let us denote by%
\begin{equation}
\varphi _{1}(z,t)=\beta \dint\limits_{0}^{e^{-t}z}u^{\beta -1}f^{\prime
}(u)du.  \label{2.5}
\end{equation}%
We obtain $\varphi _{1}(z,t)=\left( e^{-t}z\right) ^{\beta }+\frac{2\beta
a_{2}}{\beta +1}\left( e^{-t}z\right) ^{\beta +1}+...$ and we observe that%
\begin{equation}
\varphi _{1}(z,t)=z^{\beta }\varphi _{2}(z,t)  \label{2.6}
\end{equation}%
where%
\begin{equation}
\varphi _{2}(z,t)=e^{-\beta t}+\dsum\limits_{n=2}^{\infty }\frac{n\beta }{%
n+\beta -1}a_{n}e^{-\left( n+\beta -1\right) t}z^{n-1}.  \label{2.7}
\end{equation}%
The function $\varphi _{2}$ is analytic in $\mathcal{U}$ for all $t\in I,$
since%
\begin{equation*}
\overline{\underset{n\rightarrow \infty }{\lim }}\sqrt[n]{\left\vert \frac{%
n\beta }{n+\beta -1}a_{n}e^{-\left( n+\beta -1\right) t}\right\vert }=e^{-t}%
\overline{\underset{n\rightarrow \infty }{\lim }}\sqrt[n]{\left\vert
a_{n}\right\vert }.
\end{equation*}%
It is clear that if $z\in \mathcal{U},$ then $e^{-t}z\in \mathcal{U}$ for
all $t\in I$ and because $f^{\prime }(0)=1,$ there exists a disk $\mathcal{U}%
_{r_{1}},$ $0<r_{1}\leq 1$ in which $f^{\prime }(e^{-t}z)\neq 0$ for all $%
t\geq 0.$

From the analyticity of $f$ it follows that the function $\varphi _{3}$ is
also analytic in $\mathcal{U}_{r_{1}},$ where%
\begin{equation}
\varphi _{3}(z,t)=1+\left( e^{\beta mt}-e^{-\beta t}\right) z^{\beta
}h\left( e^{-t}z\right) .  \label{2.8}
\end{equation}%
We have $\varphi _{3}(0,t)=1$ and then there exists a disk $\mathcal{U}%
_{r_{2}},$ $0<r_{2}\leq r_{1}$ in which $\varphi _{3}(z,t)\neq 0$ for all $%
t\geq 0.$

Then the function%
\begin{equation}
\varphi _{4}(z,t)=\varphi _{2}(z,t)+\left( e^{\beta mt}-e^{-\beta t}\right) 
\frac{\left( g\left( e^{-t}z\right) -\alpha \right) }{\varphi _{3}(z,t)}
\label{2.9}
\end{equation}%
is also analytic in $\mathcal{U}_{r_{2}}$ and $\varphi _{4}(0,t)=(1-\alpha
)e^{\beta mt}+\alpha e^{-\beta t}.$ From $\Re \alpha <1/2,$ $\Re \beta >0$
we deduce that $\varphi _{4}(0,t)\neq 0$ for all $t\in I.$ Therefore, there
exists a disk $\mathcal{U}_{r},$ $0<r\leq r_{2}$ in which $\varphi
_{4}(0,t)\neq 0$ for all $t\in I$ and we can choose an analytic branch of $%
\left[ \varphi _{4}(z,t)\right] ^{1/\beta },$ denoted by $\varphi _{5}(z,t).$
We choose the uniform branch which is equal to $a_{1}(t)=\left[ (1-\alpha
)e^{\beta mt}+\alpha e^{-\beta t}\right] ^{1/\beta }$ at the origin, and for 
$a_{1}(t)$ we get $\underset{t\rightarrow \infty }{\lim }\left\vert
a_{1}(t)\right\vert =\infty .$ Moreover, we have $a_{1}(t)\neq 0$ for all $%
t\geq 0.$

From (\ref{2.4})-(\ref{2.9}) it results that the relation (\ref{2.4}) may be
written as%
\begin{equation}
\mathcal{L}(z,t)=z\varphi _{5}(z,t)  \label{2.10}
\end{equation}%
and hence we obtain that the function $\mathcal{L}(z,t)$ is analytic in $%
\mathcal{U}_{r},$%
\begin{equation*}
\mathcal{L}(z,t)=a_{1}(t)z+...,\text{ }\forall z\in \mathcal{U}_{r},\text{ }%
\forall t\in I.
\end{equation*}%
$\mathcal{L}(z,t)$ is an analytic function in $\mathcal{U}_{r}$ for all $%
t\in I$ and then it follows that there is a number $r_{3},$ $0<r_{3}<r$ and
a positive constant $K=K(r_{3})$ such that%
\begin{equation*}
\left\vert \frac{\mathcal{L}(z,t)}{a_{1}(t)}\right\vert <K,\text{ }\forall
z\in \mathcal{U}_{r_{3}},\text{ }t\geq 0.
\end{equation*}%
Then, by Montel's theorem, it results that $\left\{ \frac{\mathcal{L}(z,t)}{%
a_{1}(t)}\right\} _{t\geq 0}$ is a normal family in $\mathcal{U}_{r_{3}}.$

From (\ref{2.10}) we have%
\begin{equation}
\frac{\partial \mathcal{L}(z,t)}{\partial t}=z\frac{\partial \varphi
_{5}(z,t)}{\partial t}.  \label{2.11}
\end{equation}%
It is clear that $\frac{\partial \varphi _{5}(z,t)}{\partial t}$ is an
analytic function in $\mathcal{U}_{r_{3}}$ and then $\frac{\partial \mathcal{%
L}(z,t)}{\partial t}$ is also analytic function in $\mathcal{U}_{r_{3}}$.
Then, for all fixed numbers $T>0$ and $r_{4},$ $0<r_{4}<r_{3},$ there exists
a constant $K_{1}>0$ (which depends on $T$ and $r_{4}$) such that%
\begin{equation*}
\left\vert \frac{\partial \mathcal{L}(z,t)}{\partial t}\right\vert <K_{1},%
\text{ }\forall z\in \mathcal{U}_{r_{4}}\text{ and }t\in \lbrack 0,T].
\end{equation*}%
Therefore, the function $\mathcal{L}(z,t)$ is locally absolutely continuous
in $[0,\infty ),$ locally uniform with respect to $\mathcal{U}_{r_{4}}.$

Since $\frac{\partial \mathcal{L}(z,t)}{\partial t}$ is analytic in $%
\mathcal{U}_{r_{4}},$ from (\ref{2.11}) it results that there is a number $%
r_{0},$ $0<r_{0}<r_{4},$ such that $\frac{1}{z}\frac{\partial \mathcal{L}%
(z,t)}{\partial t}\neq 0,$ $\forall z\in \mathcal{U}_{r_{0}},$ and then the
function%
\begin{equation*}
p(z,t)=z\frac{\partial \mathcal{L}(z,t)}{\partial z}\diagup \frac{\partial 
\mathcal{L}(z,t)}{\partial t}
\end{equation*}%
is analytic in $\mathcal{U}_{r_{0}}$ for all $t\geq 0.$

In order to prove that the function $p(z,t)$ has an analytic extension with
positive real part in $\mathcal{U},$ to for all $t\geq 0,$ it is sufficient
to prove that the function $w(z,t)$ defined in $\mathcal{U}_{r_{0}}$ by%
\begin{equation*}
w(z,t)=\frac{p(z,t)-1}{p(z,t)+1}
\end{equation*}%
can be extended analytically in $\mathcal{U},$ $\left\vert w(z,t)\right\vert
<1$ for all $z\in \mathcal{U}$ and $t\geq 0.$

After some calculations we obtain:%
\begin{equation}
w(z,t)=\frac{2}{m+1}\mathcal{G}(z,t)-\frac{m-1}{m+1},  \label{2.122}
\end{equation}%
where%
\begin{eqnarray}
\mathcal{G}(z,t) &=&e^{-\beta \left( m+1\right) t}\left( \frac{f^{\prime
}(e^{-t}z)}{g(e^{-t}z)-\alpha }-1\right)  \notag \\
&&+\left( 1-e^{-\beta \left( m+1\right) t}\right) \left[ 2e^{-\beta
t}z^{\beta }\frac{f^{\prime }(e^{-t}z)h(e^{-t}z)}{g(e^{-t}z)-\alpha }+\frac{%
e^{-t}z}{\beta }\frac{g^{\prime }(e^{-t}z)}{g(e^{-t}z)-\alpha }\right] 
\notag \\
&&+\frac{e^{-\beta t}z^{\beta }\left( 1-e^{-\beta \left( m+1\right)
t}\right) ^{2}}{e^{-\beta \left( m+1\right) t}}  \notag \\
&&\times \left[ e^{-\beta t}z^{\beta }\frac{f^{\prime
}(e^{-t}z)h^{2}(e^{-t}z)}{g(e^{-t}z)-\alpha }+\frac{e^{-t}z}{\beta }\left( 
\frac{h(e^{-t}z)g^{\prime }(e^{-t}z)}{g(e^{-t}z)-\alpha }-h^{\prime
}(e^{-t}z)\right) \right] .  \label{2.12}
\end{eqnarray}%
for $z\in \mathcal{U}$ and $t\geq 0.$

The inequality $\left\vert w(z,t)\right\vert <1$ for all $z\in \mathcal{U}$
and $t\geq 0,$ where $w(z,t)$ defined by (\ref{2.122}), is equivalent to%
\begin{equation}
\left\vert \mathcal{G}(z,t)-\frac{m-1}{2}\right\vert <\frac{m+1}{2},\text{ }%
\forall z\in \mathcal{U}\text{ and }t\geq 0.  \label{2.13}
\end{equation}%
Define%
\begin{equation}
\mathcal{H}(z,t)=\mathcal{G}(z,t)-\frac{m-1}{2},\text{ }\forall z\in 
\mathcal{U}\text{ and }t\geq 0.  \label{2.14}
\end{equation}%
In view of (\ref{2.22}) and (\ref{2.2}), from (\ref{2.12}) and (\ref{2.14})
we have%
\begin{equation}
\left\vert \mathcal{H}(z,0)\right\vert =\left\vert \left( \frac{f^{\prime
}(z)}{g(z)-\alpha }-1\right) -\frac{m-1}{2}\right\vert <\frac{m+1}{2}.
\label{2.15}
\end{equation}%
Let $t>0,$ $z\in \mathcal{U-\{}0\}.$ In this case the function $\mathcal{H}%
(z,t)$ is analytic in $\overline{\mathcal{U}}$ because $\left\vert
e^{-t}z\right\vert \leq e^{-t}<1,$ for all $z\in \overline{\mathcal{U}}.$
Using the maximum principle for $z\in \mathcal{U}$ and $t>0$ we have%
\begin{equation*}
\left\vert \mathcal{H}(z,t)\right\vert <\underset{\left\vert \xi \right\vert
=1}{\max }\left\vert \mathcal{H}(\xi ,t)\right\vert =\left\vert \mathcal{H}%
(e^{i\theta },t)\right\vert ,
\end{equation*}%
where $\theta =\theta (t)$ is a real number.

Let $u=e^{-t}e^{i\theta }.$ We have $\left\vert u\right\vert =e^{-t}$ and $%
e^{-\beta \left( m+1\right) t}=\left( e^{-t}\right) ^{\beta \left(
m+1\right) }=\left\vert u\right\vert ^{\beta \left( m+1\right) }.$ From (\ref%
{2.12}), we have 
\begin{eqnarray*}
\left\vert \mathcal{G}(e^{i\theta },t)\right\vert &=&\left\vert \left\vert
u\right\vert ^{\beta \left( m+1\right) }\left( \frac{f^{\prime }(u)}{%
g(u)-\alpha }-1\right) \right. \\
&&+\left( 1-\left\vert u\right\vert ^{\beta \left( m+1\right) }\right) \left[
\frac{2u^{\beta }f^{\prime }(u)h(u)}{g(u)-\alpha }+\frac{u}{\beta }\frac{%
g^{\prime }(u)}{g(u)-\alpha }\right] \\
&&+\frac{u^{\beta }\left( 1-\left\vert u\right\vert ^{\beta \left(
m+1\right) }\right) ^{2}}{\left\vert u\right\vert ^{\beta \left( m+1\right) }%
} \\
&&\left. \times \left[ \frac{u^{\beta }f^{\prime }(u)h^{2}(u)}{g(u)-\alpha }+%
\frac{u}{\beta }\left( \frac{h(u)g^{\prime }(u)}{g(u)-\alpha }-h^{\prime
}(u)\right) \right] -\frac{m-1}{2}\right\vert .
\end{eqnarray*}%
Since $u\in \mathcal{U}$, the inequality (\ref{2.2}) implies that%
\begin{equation}
\left\vert \mathcal{H}(e^{i\theta },t)\right\vert \leq \frac{m+1}{2},
\label{2.16}
\end{equation}%
and from (\ref{2.15}) and (\ref{2.16}) it follows that the inequality (\ref%
{2.13})%
\begin{equation*}
\left\vert \mathcal{H}(z,t)\right\vert =\left\vert \mathcal{G}(z,t)-\frac{m-1%
}{2}\right\vert <\frac{m+1}{2}
\end{equation*}%
is satisfied for all $z\in \mathcal{U}$ and $t\in I.$ Therefore $\left\vert
w(z,t)\right\vert <1,$ for all $z\in \mathcal{U}$ and $t\geq 0.$

Since all the conditions of Theorem \ref{t1} are satisfied, we obtain that
the function $\mathcal{L}(z,t)$ has an analytic and univalent extension to
the whole unit disk $\mathcal{U},$ for all $t\in I.$ For $t=0$ we have $%
\mathcal{L}(z,0)=\mathcal{F}_{\beta }(z),$ for $z\in \mathcal{U}$ and
therefore, the function $\mathcal{F}_{\beta }(z)$ is analytic and univalent
in $\mathcal{U}.$
\end{proof}

For $g=f^{\prime }$ in Theorem \ref{t3}, we obtain another univalence
criterion as follows.

\begin{corollary}
\label{c1}\textit{Let }$m$\textit{\ be a positive real number and let }$%
\alpha ,$ $\beta $\textit{\ be complex numbers such that }$\Re \alpha <1/2,$ 
$\Re \beta >0$ and$\;f\in \mathcal{A}.$\textit{\ Let }$h$\textit{\ be an
analytic functions in }$\mathcal{U},$ $h(z)=c_{0}+c_{1}z+...$\textit{. If
the following inequalities}%
\begin{equation}
\left\vert \frac{f^{\prime }(z)}{f^{\prime }(z)-\alpha }-\frac{m+1}{2}%
\right\vert <\frac{m+1}{2},  \label{2.17}
\end{equation}%
\textit{and}%
\begin{equation*}
\left\vert \left( \frac{f^{\prime }(z)}{f^{\prime }(z)-\alpha }-1\right)
\left\vert z\right\vert ^{\beta (m+1)}\right.
\end{equation*}%
\begin{equation*}
+\left( 1-\left\vert z\right\vert ^{\beta \left( m+1\right) }\right) \left[
2z^{\beta }\frac{f^{\prime }(z)h(z)}{f^{\prime }(z)-\alpha }+\frac{1}{\beta }%
\frac{zf^{\prime \prime }(z)}{f^{\prime }(z)-\alpha }\right]
\end{equation*}%
\begin{equation}
\left. +\frac{z^{\beta +1}\left( 1-\left\vert z\right\vert ^{\beta \left(
m+1\right) }\right) ^{2}}{\left\vert z\right\vert ^{\beta \left( m+1\right) }%
}\left[ \frac{z^{\beta -1}f^{\prime }(z)h^{2}(z)}{f^{\prime }(z)-\alpha }+%
\frac{1}{\beta }\left( \frac{f^{\prime \prime }(z)h(z)}{f^{\prime
}(z)-\alpha }-h^{\prime }(z)\right) \right] -\frac{m-1}{2}\right\vert \leq 
\frac{m+1}{2}  \label{2.18}
\end{equation}%
are true for all $z\in \mathcal{U},$ then the function $\mathcal{F}_{\beta
}(z)$ defined by (\ref{2.3}) is analytic and univalent\textit{\ in }$%
\mathcal{U},$\textit{\ where the principal branch is intended.}
\end{corollary}

If we choose $h=f^{\prime \prime }$ in Corollary \ref{c1}, we have another
univalence criterion as follows.

\begin{corollary}
\label{c2}\textit{Let }$m$\textit{\ be a positive real number and let }$%
\alpha ,$ $\beta $\textit{\ be complex numbers such that }$\Re \alpha <1/2,$ 
$\Re \beta >0$\textit{\ }and$\;f\in \mathcal{A}.$\textit{\ If the following
inequalities }%
\begin{equation}
\left\vert \frac{f^{\prime }(z)}{f^{\prime }(z)-\alpha }-\frac{m+1}{2}%
\right\vert <\frac{m+1}{2},  \label{2.19}
\end{equation}%
\textit{and}%
\begin{equation*}
\left\vert \left( \frac{f^{\prime }(z)}{f^{\prime }(z)-\alpha }-1\right)
\left\vert z\right\vert ^{\beta (m+1)}\right.
\end{equation*}%
\begin{equation*}
+z\left( 1-\left\vert z\right\vert ^{\beta \left( m+1\right) }\right) \left[ 
\frac{f^{\prime \prime }(z)}{f^{\prime }(z)-\alpha }\left( 2z^{\beta
-1}f^{\prime }(z)+\frac{1}{\beta }\right) \right]
\end{equation*}%
\begin{equation}
\left. +\frac{z^{\beta +1}\left( 1-\left\vert z\right\vert ^{\beta \left(
m+1\right) }\right) ^{2}}{\left\vert z\right\vert ^{\beta \left( m+1\right) }%
}\left[ \frac{\left( f^{\prime \prime }(z)\right) ^{2}}{f^{\prime
}(z)-\alpha }\left( z^{\beta -1}f^{\prime }(z)+\frac{1}{\beta }\right) -%
\frac{1}{\beta }f^{\prime \prime \prime }(z)\right] -\frac{m-1}{2}%
\right\vert \leq \frac{m+1}{2}  \label{2.20}
\end{equation}%
are true for all $z\in \mathcal{U},$ then the function $\mathcal{F}_{\beta
}(z)$ defined by (\ref{2.3}) is analytic and univalent\textit{\ in }$%
\mathcal{U},$\textit{\ where the principal branch is intended.}
\end{corollary}

\begin{corollary}
\label{c33}\textit{Let }$m$\textit{\ be a positive real number and let }$%
\alpha ,$ $\beta $\textit{\ be complex numbers such that }$\Re \alpha <1/2,$ 
$\Re \beta >0$ and$\;f\in \mathcal{A}.$ \textit{If the following
inequalities }%
\begin{equation}
\left\vert \frac{f^{\prime }(z)}{f^{\prime }(z)-\alpha }-\frac{m+1}{2}%
\right\vert <\frac{m+1}{2},  \label{2.211}
\end{equation}%
\textit{and}%
\begin{equation}
\left\vert \left( \frac{f^{\prime }(z)}{f^{\prime }(z)-\alpha }-1\right)
\left\vert z\right\vert ^{\beta (m+1)}+\left( 1-\left\vert z\right\vert
^{\beta \left( m+1\right) }\right) \left[ \frac{1}{\beta }\frac{zf^{\prime
\prime }(z)}{f^{\prime }(z)-\alpha }\right] -\frac{m-1}{2}\right\vert \leq 
\frac{m+1}{2}  \label{2.212}
\end{equation}%
are true for all $z\in \mathcal{U},$ then the function $\mathcal{F}_{\beta
}(z)$ defined by (\ref{2.3}) is analytic and univalent\textit{\ in }$%
\mathcal{U},$\textit{\ where the principal branch is intended.}
\end{corollary}

\begin{proof}
It results from Corollary \ref{c1} with $g=f^{\prime }$ and $h=0$.
\end{proof}

If we consider $g(z)=f^{\prime },$ $h(z)=-\frac{1}{2}\frac{f^{\prime \prime }%
}{f^{\prime }},$ $\alpha =0,$ $\beta =1$ in Theorem \ref{t3}, we obtain
another univalence criterion as follows.

\begin{corollary}
\label{c3*}\textit{Let }$m$\textit{\ be a positive real number }and$\;f\in 
\mathcal{A}.$\textit{\ If the following inequality}%
\begin{equation}
\left\vert \frac{z^{2}\left( 1-\left\vert z\right\vert ^{m+1}\right) ^{2}}{%
\left\vert z\right\vert ^{m+1}}\left( \frac{1}{2}\left\{ f;z\right\} \right)
-\frac{m-1}{2}\right\vert \leq \frac{m+1}{2}  \label{2.12*}
\end{equation}%
where%
\begin{equation*}
\left\{ f;z\right\} =\left( \frac{f^{\prime \prime }(z)}{f^{\prime }(z)}%
\right) ^{\prime }-\frac{1}{2}\left( \frac{f^{\prime \prime }(z)}{f^{\prime
}(z)}\right) ^{2}
\end{equation*}%
is true for all $z\in \mathcal{U},$ then the function $f(z)$ is analytic and
univalent\textit{\ in }$\mathcal{U},$\textit{\ where the principal branch is
intended.}
\end{corollary}

Setting $\alpha =0$ in Corollary \ref{c33} we have another univalence
criterion as follows.

\begin{corollary}
\label{c333}\textit{Let }$m$\textit{\ be a positive real number and let} $%
\beta $\textit{\ be complex number such that }$\Re \beta >0$ and$\;f\in 
\mathcal{A}.$ \textit{If the following inequality}%
\begin{equation}
\left\vert \frac{\left( 1-\left\vert z\right\vert ^{\beta \left( m+1\right)
}\right) }{\beta }\left( \frac{zf^{\prime \prime }(z)}{f^{\prime }(z)}%
\right) -\frac{m-1}{2}\right\vert \leq \frac{m+1}{2}  \label{2.210}
\end{equation}%
is true for all $z\in \mathcal{U},$ then the function $\mathcal{F}_{\beta
}(z)$ defined by (\ref{2.3}) is analytic and univalent\textit{\ in }$%
\mathcal{U},$\textit{\ where the principal branch is intended.}
\end{corollary}

\begin{corollary}
\label{334}\textit{Let }$m$\textit{\ be a positive real number and let} $%
\beta $\textit{\ be complex number with }$\Re \beta >0$ and$\;f\in \mathcal{A%
}.$ \textit{If the following inequality}%
\begin{equation}
\left\vert \frac{\left( 1-\left\vert z\right\vert ^{\left( m+1\right) \Re
\beta }\right) }{\Re \beta }\left( \frac{zf^{\prime \prime }(z)}{f^{\prime
}(z)}\right) \right\vert \leq 1  \label{2.212*}
\end{equation}%
is true for all $z\in \mathcal{U},$ then the function $\mathcal{F}_{\beta
}(z)$ defined by (\ref{2.3}) is analytic and univalent\textit{\ in }$%
\mathcal{U},$\textit{\ where the principal branch is intended.}
\end{corollary}

\begin{proof}
It can be proved (see \cite{Pas}) that for $z\in \mathcal{U}\backslash
\left\{ 0\right\} ,$ $\Re \beta >0$ and $m\in 
\mathbb{R}
_{+}$%
\begin{equation*}
\left\vert \frac{1-|z|^{(m+1)\beta }}{\beta }\right\vert \leq \frac{%
1-|z|^{(m+1)\Re \beta }}{\Re \beta }.
\end{equation*}%
For $m\geq 1$, we have%
\begin{eqnarray*}
&&\left\vert \frac{1-|z|^{(m+1)\beta }}{\beta }\left( \frac{zf^{\prime
\prime }(z)}{f^{\prime }(z)}\right) -\frac{m-1}{2}\right\vert \leq
\left\vert \frac{1-|z|^{(m+1)\beta }}{\beta }\left( \frac{zf^{\prime \prime
}(z)}{f^{\prime }(z)}\right) \right\vert +\frac{m-1}{2} \\
&\leq &\frac{1-|z|^{(m+1)\Re \beta }}{\Re \beta }\left\vert \frac{zf^{\prime
\prime }(z)}{f^{\prime }(z)}\right\vert +\frac{m-1}{2} \\
&\leq &1+\frac{m-1}{2}=\frac{m+1}{2}.
\end{eqnarray*}%
Since inequalities (\ref{2.22}) and (\ref{2.2}) are satisfied, making use of
Theorem \ref{t3}, we can conclude that the function $\mathcal{F}_{\beta }$
is analytic and univalent in $\mathcal{U}$.
\end{proof}

Putting $g(z)=\left( \frac{f(z)}{z}\right) ^{2},$ $h(z)=0,$ $\alpha =0,$ in
Theorem \ref{t3}, we get the univalence criterion as follows.

\begin{corollary}
\label{c34}\textit{Let }$m$\textit{\ be a positive real number and let} $%
\beta $\textit{\ be complex number such that }$\Re \beta >0$ and$\;f\in 
\mathcal{A}.$ \textit{If the following inequalities}%
\begin{equation}
\left\vert \frac{z^{2}f^{\prime }(z)}{f^{2}(z)}-\frac{m+1}{2}\right\vert <%
\frac{m+1}{2},  \label{2.213}
\end{equation}%
\textit{and}%
\begin{equation}
\left\vert \left( \frac{z^{2}f^{\prime }(z)}{f^{2}(z)}-1\right) \left\vert
z\right\vert ^{\beta (m+1)}+\frac{2\left( 1-\left\vert z\right\vert ^{\beta
\left( m+1\right) }\right) }{\beta }\left( \frac{zf^{\prime }(z)}{f(z)}%
-1\right) -\frac{m-1}{2}\right\vert \leq \frac{m+1}{2}  \label{2.214}
\end{equation}%
are true for all $z\in \mathcal{U},$ then the function $\mathcal{F}_{\beta
}(z)$ defined by (\ref{2.3}) is analytic and univalent\textit{\ in }$%
\mathcal{U},$\textit{\ where the principal branch is intended.}
\end{corollary}

\begin{corollary}
\label{c3}\textit{Let }$m$\textit{\ be a positive real number }and$\;f\in 
\mathcal{A}.$\textit{\ If the following inequality}%
\begin{eqnarray}
&&\left\vert z\left( 1-\left\vert z\right\vert ^{m+1}\right) \left(
2f^{\prime \prime }(z)+\frac{f^{\prime \prime }(z)}{f^{\prime }(z)}\right)
\right.  \notag \\
&&\left. +\frac{z^{2}\left( 1-\left\vert z\right\vert ^{m+1}\right) ^{2}}{%
\left\vert z\right\vert ^{m+1}}\left( \frac{\left( f^{\prime \prime
}(z)\right) ^{2}}{f^{\prime }(z)}+\left( f^{\prime \prime }(z)\right)
^{2}-f^{\prime \prime \prime }(z)\right) -\frac{m-1}{2}\right\vert  \notag \\
&\leq &\frac{m+1}{2}  \label{2.21}
\end{eqnarray}%
is true for all $z\in \mathcal{U},$ then the function $f(z)$ is analytic and
univalent\textit{\ in }$\mathcal{U},$\textit{\ where the principal branch is
intended.}
\end{corollary}

\begin{proof}
It results from Corollary \ref{c2} with $\alpha =0,$ $\beta =1.$
\end{proof}

\begin{remark}
\label{r1}(1) Putting $g(z)=f^{\prime }(z),$ $h(z)=0,$ $\alpha =0,$ $\beta
=m=1$ in Theorem \ref{t3}, we have Becker's criterion \cite{Bec}.

(2) If we consider $g(z)=f^{\prime }(z),$ $h(z)=-\frac{1}{2}\frac{f^{\prime
\prime }(z)}{f^{\prime }(z)},$ $\alpha =0,$ $\beta =m=1$ in Theorem \ref{t3}%
, we obtain the univalence criterion due to Nehari \cite{Neh}.

(3) Setting $g(z)=\left( \frac{f(z)}{z}\right) ^{2},$ $h(z)=\frac{1}{z}-%
\frac{f(z)}{z^{2}},$ $\alpha =0,$ $\beta =m=1$ in Theorem \ref{t3}, we get
the univalence criterion due to Ozaki-Nunokawa \cite{Oz-Nu}.

(4) For $g(z)=f^{\prime }(z),$ $h(z)=\frac{1}{z}-\frac{f(z)}{f(z)},$ $\alpha
=0,$ $\beta =m=1$ in Theorem \ref{t3}, we arrive at Goluzin's criterion for
univalence \cite{Gol}.

(5) For $m=1$ in Corollary \ref{334}, we obtain the univalence criterion due
to Pascu \cite{Pas}.

(6) If we consider $g(z)=f^{\prime }(z),$ $h(z)=0,$ $\beta =1$ in Theorem %
\ref{t3}, we have results of Raducanu et al. \cite{Radu3}.

(7) Putting $\alpha =0,$ $\beta =m=1$ in Theorem \ref{t3}, we get the
univalence criterion due to Ovesea-Tudor and Owa \cite{Ove}.
\end{remark}

\begin{example}
\label{e1}Let the function%
\begin{equation}
f(z)=\frac{z}{1-\frac{z^{2}}{2}}.  \label{2.25}
\end{equation}

Then $f$ is univalent in $\mathcal{U}$ and the function%
\begin{equation}
\mathcal{F}_{2}(z)=\left( 2\dint\limits_{0}^{z}uf^{\prime }(u)du\right) ^{%
\frac{1}{2}}  \label{2.255}
\end{equation}%
is analytic and univalent in $\mathcal{U}$.
\end{example}

\begin{proof}
From equality (\ref{2.213}) for $m=1,$ we have%
\begin{equation}
\frac{z^{2}f^{\prime }(z)}{f^{2}(z)}-1=\frac{z^{2}}{2}.  \label{2.25*}
\end{equation}

It is clear that the condition (\ref{2.213}) of the Corollary \ref{c34} is
satisfied for $m=1,$ and then the function $f$ is univalent in $\mathcal{U}. 
$

Taking into account (\ref{2.25*}), the condition (\ref{2.214}) of Corollary %
\ref{c34} becomes for $\beta =2,$ $m=1,$%
\begin{eqnarray*}
\left\vert \frac{z^{2}}{2}\left\vert z\right\vert ^{4}+\left( 1-\left\vert
z\right\vert ^{4}\right) \frac{2z^{2}}{2-z^{2}}\right\vert &\leq &\frac{%
\left\vert z\right\vert ^{6}}{2}+2\left( 1-\left\vert z\right\vert
^{4}\right) \left\vert z\right\vert ^{2} \\
&=&\frac{1}{2}\left( 4\left\vert z\right\vert ^{2}-3\left\vert z\right\vert
^{6}\right) <1
\end{eqnarray*}%
because the greatest value of the function $g(x)=4x^{2}-3x^{6},$ for $x\in %
\left[ 0,1\right] $ is taken for $x=\sqrt{\frac{2}{3}}$ and $g(\sqrt{\frac{2%
}{3}})=\frac{24}{27}.$ Therefore the function $\mathcal{F}_{2}(z)$ defined
by (\ref{2.255}) is analytic and univalent in $\mathcal{U}$.%
\begin{equation*}
\FRAME{itbpFU}{3.288in}{2.3307in}{0in}{\Qcb{Figure 1: $f(z)=\frac{z}{1-\frac{%
z^{2}}{2}}$}}{}{Figure}{\special{language "Scientific Word";type
"GRAPHIC";display "USEDEF";valid_file "T";width 3.288in;height
2.3307in;depth 0in;original-width 3.7395in;original-height 2.0626in;cropleft
"0";croptop "1";cropright "1";cropbottom "0";tempfilename
'M1HBSF00.wmf';tempfile-properties "XPR";}}\text{ \ \ \ \ \ \ }\FRAME{itbpFU%
}{2.5408in}{2.5313in}{0in}{\Qcb{Figure 2: $\mathcal{F}_{2}(z)=\left( 4%
\protect\dint\limits_{0}^{z}\frac{2+u^{2}}{\left( 2-u^{2}\right) ^{2}}%
du\right) ^{\frac{1}{2}}$}}{}{Figure}{\special{language "Scientific
Word";type "GRAPHIC";display "USEDEF";valid_file "T";width 2.5408in;height
2.5313in;depth 0in;original-width 2.4163in;original-height 4.4996in;cropleft
"0";croptop "1";cropright "1";cropbottom "0";tempfilename
'M1HBIC02.wmf';tempfile-properties "XPR";}}
\end{equation*}
\end{proof}

\section{Quasiconformal Extension Criterion}

In this section we will obtain the univalence condition given in Theorem \ref%
{t3} to a quasiconformal extension criterion.

\begin{theorem}
\label{t4}\textit{Let }$m$\textit{\ be a positive real number and let }$%
\alpha ,$ $\beta $\textit{\ be complex numbers such that }$\Re \alpha <1/2,$ 
$\Re \beta >0$\textit{,}$\;f\in \mathcal{A}$ and $k\in \lbrack 0,1).$\textit{%
\ Let} $g$\textit{\ and }$h$\textit{\ be two analytic functions in }$%
\mathcal{U},$ $g(z)=1+b_{1}z+...,$ $h(z)=c_{0}+c_{1}z+...$\textit{. If the
following inequalities}%
\begin{equation}
\left\vert \frac{f^{\prime }(z)}{g(z)-\alpha }-\frac{m+1}{2}\right\vert <k%
\frac{m+1}{2},  \label{2.26}
\end{equation}%
\textit{and}%
\begin{equation*}
\left\vert \left( \frac{f^{\prime }(z)}{g(z)-\alpha }-1\right) \left\vert
z\right\vert ^{\beta (m+1)}\right.
\end{equation*}%
\begin{equation*}
+\left( 1-\left\vert z\right\vert ^{\beta \left( m+1\right) }\right) \left[
2z^{\beta }\frac{f^{\prime }(z)h(z)}{g(z)-\alpha }+\frac{1}{\beta }\frac{%
zg^{\prime }(z)}{g(z)-\alpha }\right]
\end{equation*}%
\begin{equation}
\left. +\frac{z^{\beta +1}\left( 1-\left\vert z\right\vert ^{\beta \left(
m+1\right) }\right) ^{2}}{\left\vert z\right\vert ^{\beta \left( m+1\right) }%
}\left[ \frac{z^{\beta -1}f^{\prime }(z)h^{2}(z)}{g(z)-\alpha }+\frac{1}{%
\beta }\left( \frac{g^{\prime }(z)h(z)}{g(z)-\alpha }-h^{\prime }(z)\right) %
\right] -\frac{m-1}{2}\right\vert \leq k\frac{m+1}{2}  \label{2.27}
\end{equation}%
is true for all $z\in \mathcal{U},$ then the function $\mathcal{F}_{\beta
}(z)$ given by (\ref{2.3}) has a $k-$quasiconformal extension to $%
\mathbb{C}
.$
\end{theorem}

\begin{proof}
Set%
\begin{equation}
\mathcal{L}(z,t)=\left[ \beta \dint\limits_{0}^{e^{-t}z}u^{\beta
-1}f^{\prime }(u)du+\frac{\left( e^{\beta mt}-e^{-\beta t}\right) z^{\beta
}\left( g\left( e^{-t}z\right) -\alpha \right) }{1+\left( e^{\beta
mt}-e^{-\beta t}\right) z^{\beta }h\left( e^{-t}z\right) }\right] ^{1/\beta }
\label{3.19}
\end{equation}%
In the proof of Theorem \ref{t3} has been shown that the function $\mathcal{L%
}(z,t)$ given by (\ref{3.19}) is a subordination chain in $\mathcal{U}.$
Then we have%
\begin{eqnarray}
\left\vert \frac{p(z,t)-1}{p(z,t)+1}\right\vert &=&\left\vert \frac{2}{m+1}%
\left\{ e^{-\beta \left( m+1\right) t}\left( \frac{f^{\prime }(e^{-t}z)}{%
g(e^{-t}z)-\alpha }-1\right) \right. \right.  \notag \\
&&+\left( 1-e^{-\beta \left( m+1\right) t}\right) \left[ 2e^{-\beta
t}z^{\beta }\frac{f^{\prime }(e^{-t}z)h(e^{-t}z)}{g(e^{-t}z)-\alpha }+\frac{%
e^{-t}z}{\beta }\frac{g^{\prime }(e^{-t}z)}{g(e^{-t}z)-\alpha }\right] 
\notag \\
&&+\frac{e^{-\beta t}z^{\beta }\left( 1-e^{-\beta \left( m+1\right)
t}\right) ^{2}}{e^{-\beta \left( m+1\right) t}}  \notag \\
&&\left. \left. \times \left[ e^{-\beta t}z^{\beta }\frac{f^{\prime
}(e^{-t}z)h^{2}(e^{-t}z)}{g(e^{-t}z)-\alpha }+\frac{e^{-t}z}{\beta }\left( 
\frac{h(e^{-t}z)g^{\prime }(e^{-t}z)}{g(e^{-t}z)-\alpha }-h^{\prime
}(e^{-t}z)\right) \right] \right\} -\frac{m-1}{m+1}\right\vert  \notag \\
&\leq &k.  \label{3.20}
\end{eqnarray}%
The right hand of (\ref{3.20}) always less than or equal to $k$ from (\ref%
{2.27}) and therefore $\mathcal{F}_{\beta }$ can be extended to $k$
quasiconformal mapping to ${\mathbb{C}}$ by Theorem \ref{t1} and Theorem \ref%
{t2}.
\end{proof}

\noindent

\end{document}